\documentclass[review,sort&compress]{elsarticle}
\usepackage{graphicx}
\usepackage{amsfonts}
\usepackage{amsmath}
\usepackage{color}
\usepackage{enumerate}
\usepackage{hyperref}
\usepackage{amssymb}
\usepackage{amsthm}

\newtheorem{lemma}{Lemma}
\newtheorem{corollary}{Corollary}
\newtheorem{proposition}{Proposition}
\newtheorem{remark}{Remark}

\newtheorem{example}{Example}
\newtheorem{definition}{Definition}
\newtheorem{theorem}{Theorem}
\DeclareMathOperator*{\maximize}{maximize}
\DeclareMathOperator*{\diag}{diag}

\usepackage{calrsfs}
\DeclareMathAlphabet{\pazocal}{OMS}{zplm}{m}{n}
\renewcommand{\i}[1]{\pazocal{#1}}
\usepackage{calrsfs}
\renewcommand{\i}[1]{\pazocal{#1}}

\usepackage{mathtools}
\mathtoolsset{showonlyrefs}

\journal{Linear Algebra and its Applications}

\begin{document}

\begin{frontmatter}

\title{Kronecker weights for instability analysis of Markov jump linear systems}

\author[myfirstaddress]{Wenjie Mei\corref{mycorrespondingauthor}}
\ead{mei.wenjie.mu2@is.naist.jp}
\cortext[mycorrespondingauthor]{Corresponding author. Tel:+81-743-72-5353; fax:+81-043-72-5359}
\address[myfirstaddress]{Graduate School of Science and Technology, Nara Institute of Science and Technology, Nara, Japan}

\author[myfirstaddress]{Masaki Ogura}
\ead{oguram@is.naist.jp}

\begin{abstract}
In this paper, we analyze the instability of continuous-time Markov jump linear systems. Although there exist several effective criteria for the stability of Markov jump linear systems, there is a lack of methodologies for verifying their instability. In this paper, we present a novel criterion for the exponential mean instability of Markov jump linear systems.  The main tool of our analysis is an auxiliary Markov jump linear system, which results from taking the Kronecker products of the given system matrices and a set of appropriate matrix weights. We furthermore show that the problem of finding matrix weights for tighter instability analysis can be transformed to the spectral optimization of an affine matrix family, which can be efficiently performed by gradient-based non-smooth optimization algorithms. We confirm the effectiveness of the proposed methods by numerical examples.
\end{abstract}


\end{frontmatter}

\section{Introduction} 

The stability analysis of switched linear systems is of fundamental importance in the systems and control theory~\cite{Lin2009}. In particular, the stability analysis of Markov jump linear systems has been extensively investigated in the literature~\cite{Costa2013,Shi2015}. Markov processes allow us to efficiently model random changes in the parameters of dynamical systems. Therefore, Markov jump linear systems have several applications including  power systems~\cite{PengShi1999a}, satellite control~\cite{Fragoso2005}, cell growth~\cite{Ogura2016n}, and temporal networks~\cite{Ogura2015c}. As for the stability analysis of Markov jump linear systems, Feng et al.~\cite{Feng1992} gave a necessary and sufficient condition for the exponential mean square stability of continuous-time Markov jump linear systems. Fang and Loparo~\cite{Fang2002a} studied discrete-time Markov jump linear systems and gave a necessary and sufficient condition for the exponential mean square stability. The authors in~\cite{Bolzern2010} investigated Markov jump linear systems characterized by piecewise-constant transition rates and gave a sufficient condition for the exponential mean square stability. A necessary and sufficient condition for the exponential mean stability of positive Markov jump linear systems was presented in~\cite{Ogura2014}. 

Besides the stability analysis that we mentioned above, the \emph{instability} analysis of Markov jump systems have been of interest in the literature due to its applications in, e.g., the predator-prey models under regime-switchings~\cite{Li2011a,Bao2016}. In this context, the authors in~\cite{Yin2010} discussed stability and instability of regime-switching jump diffusion processes. Shao et al.~\cite{Shao2014} presented easily verifiable conditions for the stability and instability of nonlinear regime-switching processes. The authors in~\cite{Liu2010} investigated the persistence (instability) and extinction (stability) of a single-species model under regime-switching driven by both white noise and colored noise. Liu~\cite{Liu2015b} studied the random perturbations and analyzed the persistence and extinction of a two-species cooperation model. 

However, there is a lack of practical criteria in the literature for instability analysis that is specifically tailored to Markov jump linear systems. Although the instability analyses in~\cite{Shao2014,Li2009} for Markov jump systems with diffusions can be applied to Markov jump linear systems, their results cannot be expected to be tight. A major reason is that their analyses do not distinguish system matrices having the same spectral abscissa (i.e., the maximum real part of the eigenvalues).  An exception to this trend can be found in~\cite{Ogura2016d}, where the authors proposed instability conditions for discrete-time Markov jump linear systems without such identification of system matrices. We also remark that, although almost sure stability has been often studied in the literature~\cite{Li2011a,Bao2016,Yin2010,Liu2010,Liu2015b,Shao2014,Li2009}, the exponential mean stability is more commonly employed in the systems and control theory~\cite{Costa2013,Shi2015}.

In this paper, we present an optimization-based framework for analyzing the instability of continuous-time Markov jump linear systems. We show that, if a certain matrix weighted by another set of matrices via Kronecker products has a nonnegative spectral abscissa, then the Markov jump linear system is not exponentially mean stable. We also prove that, the larger the size of the matrix weights, the tighter our instability analysis. Furthermore, we show that we can use a gradient-based non-smooth optimization algorithm to efficiently find local optimal matrix weights for tighter instability analysis.

This paper is organized as follows. After giving mathematical preliminaries, in Section~\ref{sec2} we give a brief overview of Markov jump linear systems and present our main result on instability analysis. The proof of the main result is given in Section~\ref{sec:proof}. The monotonicity result on our instability criteria is presented in Section~\ref{sec:weight}. In Section~\ref{sec:Stochastic}, we illustrate how we can employ a gradient-based non-smooth optimization algorithm for finding locally optimal matrix weights for tighter instability analysis. 

\subsection{Mathematical preliminaries}

For a positive integer~$n$, let $I_n$ denote the $n\times n$ identity matrix. 
The Euclidean norm of a real vector~$v$ is denoted by~$\rVert v \rVert$. For a real matrix~$A$, let $\rVert A \rVert$ denote the maximum singular value of~$A$. The maximum real part of the eigenvalues of~$A$, denoted by~$\mu(A)$, is called as the spectral abscissa of~$A$. We say that $A$ is Hurwitz stable if $\mu(A) < 0$. For matrices~$A_1$, \dots,~$A_N$, let $\diag(A_1, \dotsc, A_N)$ or $\bigoplus_{i=1}^N A_i$ denote the block-diagonal matrix having the block diagonals~$A_1$, \dots, $A_N$. Let $e_i$ denote the $i$th standard unit vector in $\mathbb{R}^N$, i.e., the unit vector whose elements are all zero except the $i$th element being one. We let $\otimes$ denote the Kronecker product of matrices. Given a square matrix~$B$, the Kronecker sum of~$A$ and~$B$ is defined by $A \oplus B=A \otimes I_m+I_n \otimes B$, where $n$ and $m$ are the orders of~$A$ and $B$, respectively. It is known that the identities
\begin{equation}\label{eq:ABABAB}
e^A \otimes e^B  = e^{A \oplus B}
\end{equation}
and 
\begin{equation}\label{eq:ABAB}
\rVert A \otimes B \rVert= \rVert A \rVert \rVert B \rVert
\end{equation}
hold true (see, e.g.,~\cite{Brewer1978}). We also have 
\begin{equation}\label{eq:ABCD}
(A \otimes C)(B \otimes D)=(AB)\otimes (CD)
\end{equation}
for matrices $C$ and $D$ for which the product~$CD$ is defined.

For $x \in \mathbb{R}^n$ and a positive integer $p$ the vector $x^{[p]}$ is defined~\cite{Barkin1983} as the real vector of length 
\begin{equation}
n_p=\binom{n+p-1}{p},
\end{equation}
whose elements are all the lexicographically ordered monomials of degree $p$ in $x_1$, \dots,~$x_n$. The coefficients of the monomials are chosen in such a way that $\rVert x^{[p]} \rVert= \rVert x \rVert^p$. Then, we can show that 
\begin{equation}\label{eq:norm}
    \rVert e_i \otimes x^{[p]} \rVert= \rVert x \rVert^p.
\end{equation}
Also, the following lemma holds true: 
\begin{lemma}[{\cite[Lemma~1.5]{Ogura2014}}] \label{lemma}
The set $\{x^{[p]} \colon x \in \mathbb{R}^n \}$ is a basis of $\mathbb{R}^{n_p}$. 
\end{lemma}

For an $n\times n$ real matrix~$A$, we define $A_{[p]}$ as the unique $n_p \times n_p$ matrix such that, if an $\mathbb{R}^n$-valued function $x$ satisfies $dx/dt = Ax$, then the vectorial function~$x^{[p]}$ satisfies the differential equation~${dx^{[p]}}/{dt}=A_{[p]}x^{[p]}$~(see \cite{Barkin1983}). It is known \cite{Barkin1983} that the equation 
\begin{equation}
(cA)_{[p]}=c(A)_{[p]}
\end{equation}
holds true for any real number~$c$.

\section{Instability analysis} \label{sec2}

We present the main result of this paper in this section. In Subsection~\ref{sec:MJLS}, we introduce Markov jump linear systems~\cite{Costa2013} and review their stability properties. Then, in Subsection~\ref{sec:mainResult}, we state our main result on the instability analysis of Markov jump linear systems.  

\subsection{Markov jump linear systems}\label{sec:MJLS}

Markov jump linear systems~\cite{Costa2013} is a class of switched linear systems~\cite{Lin2009} and is defined as follows:  

\begin{definition}[{\cite{Costa2013}}]
Let $A_1$, \dots, $A_N$ be $n \times n$ real matrices. Let $r$ be a continuous-time Markov process taking values in the set~$\{1, \dotsc, N\}$. Then, the following stochastic differential equation 
\begin{equation}\label{eq:MJLS}
\Sigma: \frac{dx}{dt} = A_{r(t)}x(t),\ x(0)=x_0 
\end{equation}
is called a \emph{Markov jump linear system}. 
\end{definition} 

In this paper, we are concerned with the stability property of Markov jump linear systems defined as follows: 

\begin{definition} \label{def:main}
Let $p$ be a positive integer. The Markov jump linear system~\eqref{eq:MJLS} is said to be \emph{exponentially $p$th mean stable} if there exist positive constants $C$ and $\beta$ such that
\begin{equation}\label{eq:def:pthmeanstability}
{\mathbb{E}} [\rVert x(t) \rVert^p] \leq Ce^{-\beta t}\rVert x_0\rVert^p
\end{equation} 
for all $x_0\in \mathbb{R}^n$ and $t \geq 0$.  
\end{definition}

We can find various stability criteria for Markov jump linear systems in the literature~\cite{Shi2015}. The following proposition gives a characterization of the mean square stability (i.e., the $p$th mean stability with $p=2$) in terms of the Hurwitz stability of a certain matrix.  

\begin{proposition}[{\cite{Feng1992}}]
Let $Q$ be the infinitesimal generator of the Markov process~$r$. Then, the Markov jump linear system~\eqref{eq:MJLS} is exponentially mean square stable if and only if the matrix
\begin{equation*}
Q^\top \otimes I_{n^2}+\diag \big(A_1\oplus A_1,\dotsc,A_N\oplus A_N \big)
\end{equation*}
is Hurwitz stable. 
\end{proposition}

For the specific class of Markov jump linear systems called \emph{positive Markov jump linear systems}~\cite{Ogura2014,Bolzern2014}, the following proposition gives a characterization of the $p$th mean stability for any $p$ in terms of the Hurwitz stability of yet another matrix. Recall that the Markov jump linear system~\eqref{eq:MJLS} is said to be \emph{positive} \cite{Ogura2014,Bolzern2014} if the matrices $A_1$, \dots, $A_N$ are Metzler, i.e., the matrices have nonnegative off-diagonals~\cite{Farina2000}. 

\begin{proposition}[{\cite[Theorem~5.1]{Ogura2014}}]\label{prop:ogura}
Let $p$ be a positive integer. Assume that the Markov jump linear system~\eqref{eq:MJLS} is positive. Let $Q$ be the infinitesimal generator of the Markov process~$r$. Then, the Markov jump linear system~\eqref{eq:MJLS} is exponentially $p$th mean stable if and only if the matrix
\begin{equation}\label{matrix:main}
{\i T}=Q^\top \otimes I_{n_p}+\diag \big((A_1)_{[p]},\dotsc,(A_N)_{[p]} \big)
\end{equation}
is Hurwitz stable~\cite{Ogura2014}. 
\end{proposition}

However, the stability condition in Proposition~\ref{prop:ogura} does not necessarily hold true without positivity of the system. Let us consider the following simple but yet illustrative example. 

\begin{example} \label{example:1}
Let us consider the Markov jump linear system~\eqref{eq:MJLS} with
\begin{equation}
A_1=\begin{bmatrix} 
0&-1   \\
1&0 \end{bmatrix},\ 
A_2=
\begin{bmatrix} 
0&1  \\
-1&0 \end{bmatrix}. 
\end{equation}
and the $\{1, 2\}$-valued Markov process~$r$ having the infinitesimal generator
\begin{equation}
Q= \begin{bmatrix} 
    -1&1  \\
    1&-1 \end{bmatrix}. 
\end{equation}
Since $A_1$ and $A_2$ are skew-symmetric, the matrix exponentials~$e^{A_1t}$ and~$e^{A_2t}$ are unitary for all $t \geq 0$. This implies that $\rVert x(t) \rVert = \rVert x_0 \rVert$ for all $t\geq 0$. Therefore, the Markov jump linear system~\eqref{eq:MJLS} is not exponentially $p$th mean stable for any $p$. On the other hand, if $p=1$, then the matrix~$\i T$ given in~\eqref{matrix:main} has the spectral abscissa~$-1$ and, therefore, is Hurwitz stable. Hence, the stability of the Markov jump linear system and the matrix~$\i T$ are not compatible. 
\end{example}

In the next subsection, we fill in the gap between the exponential mean stability of the Markov jump linear system~\eqref{eq:MJLS} and the spectral abscissa of the matrix~$\i T$, and present a novel instability criterion for Markov jump linear systems.

\subsection{Main result}\label{sec:mainResult}

In order to state the main result of this paper, we introduce another class of switched linear systems called \emph{deterministic switched linear systems}. 

\begin{definition}[\cite{Lin2009}]
Let $A_1$ ,\dots, $A_N$ be $n \times n$ real matrices. Let $\sigma$ be a piecewise constant function taking values in the set~$\{1, \dotsc, N\}$. Then, the differential equation
\begin{equation}\label{eq:Sigmad}
\frac{dx}{dt} = A_{\sigma(t)}x(t),\ x(0)=x_0 
\end{equation}
is called a \emph{deterministic switched linear system}. The deterministic switched linear system \eqref{eq:Sigmad} is said to be \emph{stable} if  there exists $C>0$ such that
\begin{equation}\label{eq:stableDeterministicSwitchedSystem}
\rVert x(t) \rVert \leq C \rVert x_0\rVert
\end{equation} 
for all $x_0$, $t \geq 0$, and the switching signal $\sigma$. 
\end{definition}

The following theorem presents a novel instability criterion for Markov jump linear systems and is the main result of this paper. The proof of the theorem is presented in Section~\ref{sec:proof}. 

\begin{theorem}\label{thm:main}
Let $m$ be a positive integer. Let $W_1$, \dots, $W_N$ be $m\times m$ real matrices. Assume that the deterministic switched linear system
\begin{equation}\label{eq:sls:W}
\Sigma_{d, W}:\frac{dy}{dt} = W_{\sigma(t)}y(t), \ y(0)=y_0
\end{equation}
is stable. If the matrix
\begin{equation}\label{eq:hattau1}
\hat{\i T} =Q^\top \otimes I+\diag \big( ({W}_1 \oplus A_1)_{[p]},\dotsc,({W}_N \oplus A_N)_{[p]} \big)
\end{equation}
is not Hurwitz stable, then the Markov jump linear system~$\Sigma$ is not exponentially $p$th mean stable.
\end{theorem} 

Let us see an example. 

\begin{example} \label{example:2}
Let us apply Theorem~\ref{thm:main} to the  Markov jump linear system considered in Example~\ref{example:1}. Let $p=1$ and set $W_1=A_2$ and $W_2=A_1$. Then, we see that $\mu (\hat{\i T})=0$ and, therefore, $\hat{\i T}$ is not Hurwitz stable. Hence, Theorem~\ref{thm:main} shows that the Markov jump linear system $\Sigma$ is not mean stable, which coincides with our analysis in Example~\ref{example:1}. 
\end{example}

We remark that, although Example~\ref{example:2} illustrates that the matrix weights~$W_i$ can significantly improve our ability to analyze the stability property of Markov jump linear systems, it is not a trivial problem to systematically choose appropriate weights. This issue is discussed in Section~\ref{sec:Stochastic}. 

\section{Proof} \label{sec:proof}

We present the proof of Theorem~\ref{thm:main} in this section. We start by showing the following preliminary result:   

\begin{proposition}\label{prop:instability}
Consider the Markov jump linear system~$\Sigma$ given in \eqref{eq:MJLS}. If the matrix~${\i T}$ defined in \eqref{matrix:main} is not Hurwitz stable, then $\Sigma$ is not exponentially $p$th mean stable. 
\end{proposition}  

\begin{proof}
We prove the proposition by showing its contraposition. Assume that $\Sigma$ is exponentially $p$th mean stable. Define the stochastic process $\delta$ taking values in the set~$\{e_1, \dotsc, e_N\} \subset \mathbb{R}^N$ by
\begin{equation*}
    \delta(t)=e_{r(t)}
\end{equation*}
for all $t\geq 0$. Then, Proposition~5.3 in~\cite{Ogura2014} show that the differential equation 
\begin{equation*}
\frac{d}{dt}\mathbb{E}[\delta \otimes x^{[p]}] = {\i T} \mathbb{E}[\delta \otimes x^{[p]}]
\end{equation*}
holds true. From this differential equation, we obtain 
\begin{equation}\label{eq:proofProp}
e^{{\i T} t} \big(\delta_0 \otimes x_0^{[p]} \big) = \mathbb{E} [\delta(t) \otimes x(t)^{[p]}]
\end{equation}
for all $x_0 \in \mathbb{R}^n$ and $t \geq 0$. Since we have $\lVert \mathbb{E} [\delta(t) \otimes x(t)^{[p]}] \rVert \leq \mathbb{E} [\lVert \delta(t) \otimes x(t)^{[p]} \rVert ]= \mathbb{E}[\lVert x(t)\rVert^p]$ by equation \eqref{eq:norm}, taking the norm on the both hand sides of the equation~\eqref{eq:proofProp} shows that 
$\lVert e^{{\i T} t} (\delta_0 \otimes x_0^{[p]}) \rVert \leq \mathbb{E} [\lVert x(t)\rVert^p ]  \leq Ce^{-\beta t}\rVert x_0\rVert^p$ for some $C>0$ and $\beta>0$ by the exponential $p$th mean stability of~$\Sigma$. This inequality implies
\begin{equation}\label{eq:convergesTo0}
\lim_{t\to\infty}e^{{\i T} t} z = 0
\end{equation}
provided that $z$ is of the form $\delta_0 \otimes x_0^{[p]}$ for some $\delta_0 \in \{e_1, \dotsc, e_N\}$ and $x_0\in\mathbb{R}^n$. Since Lemma~\ref{lemma} shows that the set of such vectors spans the whole space~$\mathbb{R}^{Nn_p}$, we conclude that equation~\eqref{eq:convergesTo0} holds true for any $z\in\mathbb{R}^{Nn_p}$. This implies that the matrix~${\i T}$ is Hurwitz stable and completes the proof of the theorem. 
\end{proof}

Let us prove Theorem~\ref{thm:main}.

\begin{proof}[Proof of Theorem~\ref{thm:main}] \label{main:proof}
Assume that $\Sigma_{d, W}$ is stable. We consider the following auxiliary Markov jump linear system: 
\begin{equation}\label{eq:hatSigma}
\hat \Sigma : \frac{d\hat x}{dt} =  \big( W_{\sigma(t)} \oplus A_{\sigma(t)} \big) \hat x. 
\end{equation}
Let $\sigma$ be a sample path of the Markov process~$r$. Let $t\geq 0$ be arbitrary. Then, there exist positive numbers~$h_1$, \dots, $h_K$ as well as integers~$i_1$, \dots, $i_K$ in the set~$\{1, \dotsc, N\}$ such that 
\begin{equation}\label{eq:hatx=1}
\begin{aligned}
\hat x(t)&= e^{({W}_{i_1} \oplus A_{i_1})h_1}  \dotsm e^{({W}_{i_K} \oplus A_{i_K})h_K} \hat x(0) \\
&= \left[\prod_{f=1}^K e^{({W}_{i_f} \oplus A_{i_f})h_f} \right] \hat x(0).
\end{aligned}       
\end{equation}
Since
\begin{equation}
\begin{aligned}
(W_{i_f} \oplus A_{i_f})h_f &=(I_m \otimes A_{i_f}+W_{i_f} \otimes I_n)h_f \\
&=I_m \otimes (h_f A_{i_f})+(h_f W_{i_f}) \otimes I_n \\
&=(h_f A_{i_f}) \oplus (h_f W_{i_f}), 
\end{aligned} 
\end{equation}
for all $f=1, \dotsc, K$, equations~\eqref{eq:ABABAB} and \eqref{eq:hatx=1} show that 
\begin{equation}\label{eq:hatx=2}
\begin{aligned}
\hat x(t)
&=\left[\prod_{f=1}^{K} \biggl (e^{W_{i_f}h_f} \otimes e^{A_{i_f}h_f} \biggr)\right] \hat x(0)\\
&=\left[\biggl(\prod_{f=1}^{K} e^{W_{i_f}h_f}\biggr) \otimes \biggl(\prod_{f=1}^{K} e^{A_{i_f}h_f}\biggr)\right] \hat x(0),
\end{aligned}       
\end{equation}
where we used equation~\eqref{eq:ABCD} in the last equality. 

Since $\hat{\i T}$ is assumed to be not Hurwitz stable, Proposition~\ref{prop:instability} shows that $\hat \Sigma$ is not exponentially $p$th mean stable. Therefore, there exists a nonzero vector~$\hat x_0 \in {\mathbb R}^{mn}$ such that, if $\hat x(0)=\hat x_0$, then 
\begin{equation}\label{eq:||hat x||>=...}
{\mathbb{E}} [\rVert \hat x(t) \rVert^p] \geq 1
\end{equation}
for all $t\geq 0$. Let $e_1 \dots, e_m$ be the canonical basis of~${\mathbb R}^{m}$. Then, we can take vectors~$y_1, \dotsc, y_m \in {\mathbb{R}}^{n}$ such that 
\begin{equation}\label{eq:hatx_0=...}
\hat x_0 = e_1 \otimes y_1 + \cdots + e_m \otimes y_m. 
\end{equation}
From equations~\eqref{eq:hatx=2} and \eqref{eq:hatx_0=...}, we obtain 
\begin{equation}
\begin{aligned}
 \hat x(t) =& \sum_{j=1}^m \left[\left( \biggl (\prod_{f=1}^{K} e^{W_{i_f}h_f} \biggr) \otimes \biggl(\prod_{f=1}^{K} e^{A_{i_f}h_f} \biggr) \right)(e_j \otimes y_j)\right] \\
=&\sum_{j=1}^m \left[\left( \biggl (\prod_{f=1}^{K} e^{W_{i_f}h_f} \biggr )e_j \right) \otimes \left(\biggl (\prod_{f=1}^{K} e^{A_{i_f}h_f} \biggr)y_j \right) \right].
\end{aligned}
\end{equation}
Taking the norms in the both hand sides of this equation and using the triangle inequality, we obtain  
\begin{equation} \label{triangle:ineq}
\begin{aligned}
\rVert \hat x(t) \rVert ^p
& \leq
\left(\sum_{j=1}^m \bigg \rVert \bigg(\prod_{f=1}^{K} e^{W_{i_f}h_f} \bigg)e_j \bigg \rVert \bigg \rVert \bigg(\prod_{f=1}^{K} e^{A_{i_f}h_f} \bigg)y_j \bigg \rVert  \right)^p, 
\end{aligned}
\end{equation}
where we used equation~\eqref{eq:ABAB}. Because $\Sigma_{d, W}$ is stable, there exists $C >0$ such that $\lVert (\prod_{f=1}^{K} e^{W_{i_f}h_f} )e_j \rVert \leq C$ for all $j$. Therefore, applying the Cauchy-Schwartz Inequality
to the right-hand side of inequality~\eqref{triangle:ineq}, we obtain 
\begin{equation}\label{main:i}
\begin{aligned}
\rVert \hat x(t) \rVert ^p
& \leq 
\Bigg[\sum_{j=1}^m \bigg \rVert \bigg(\prod_{f=1}^{K} e^{W_{i_f}h_f} \bigg)e_j \bigg \rVert^p \bigg] 
\bigg[\sum_{j=1}^m \bigg \rVert \bigg(\prod_{f=1}^{K} e^{A_{i_f}h_f} \bigg)y_j \bigg \rVert^p \bigg] \\
& \leq
m C^p \sum_{j=1}^m \rVert x(t; y_j) \rVert^p,
\end{aligned}
\end{equation}
where $x(\cdot; y_j)$ denote the solution of the Markov jump linear system~$\Sigma$ with the initial condition~$x(0) = y_j$. 

Now, since the sample path~$\sigma$ was arbitrary, taking mathematical expectations in the inequality~\eqref{main:i} shows ${\mathbb{E}}[\rVert \hat x(t) \rVert^p] \leq m C^p  \sum_{j=1}^m  {\mathbb{E}} [\rVert x(t; y_j) \rVert^p]$.  From this inequality and \eqref{eq:||hat x||>=...}, we obtain $\sum_{j=1}^m {\mathbb{E}}[\rVert x(t; y_j)\rVert^p] \geq m^{-1} C^{-p}$. This inequality shows the existence of $j \in \{1,2,\dotsc,m\}$ such that ${\mathbb{E}}[\rVert x(t;y_j) \rVert^p]$ does not converge to $0$ as $t$ tends to $\infty$. Therefore, $\Sigma$ is not exponentially $p$th mean stable and this completes the proof of the theorem.
\end{proof}

\section{Monotonicity} \label{sec:weight}

In order to fully benefit from the introduction of the matrix weights~$W_1$, \dots, $W_N$ in the matrix~$\hat{\i T}$, we should be able to find the weights that maximize the spectral abscissa of~$\hat{\i T}$, i.e., the matrix weights solving the following optimization problem:
\begin{equation}
\begin{aligned}
\maximize_{W_1,\,\dotsc,\,W_N \in \mathbb{R}^{m\times m}} \quad & \mu(\hat{\i T})
\\
\text{subject to\ \ \ \,} \quad & \mbox{$\Sigma_{d, W}$ is stable.}
\end{aligned}
\end{equation}
Let us denote the solution of the optimization problem by $\i U_m$, i.e., define  
\begin{equation}\label{eq:def:U_m}
{\i{U}}_m = \sup_{W_1,\,\dotsc,\,W_N\in \mathbb{R}^{m\times m}} \big \{
\mu(\hat {\i T}) \colon \mbox{$\Sigma_{d, W}$ is stable} \big \}, 
\end{equation}
where we place the subscript~$m$ to emphasize that the quantity depends on the order~$m$ of the matrix weights. Then, the following corollary immediately follows from Theorem~\ref{thm:main}. 

\begin{corollary}\label{cor:}
If there exists a positive integer~$m$ such that $\i U_m\geq 0$, then the Markov jump linear system~$\Sigma$ is not exponentially $p$th mean stable. 
\end{corollary}

In this section, we prove the following monotonicity result. Specifically, the following proposition shows that, the larger the order of the matrix weights, the tighter our instability analysis. 

\begin{proposition} \label{prop:monotonicity}
Let $m$ and  $m'$ be positive integers. If $m<m'$, then ${\i{U}}_m \leq {\i{U}}_{m'}$. 
\end{proposition}

For the proof of this proposition, we present the following lemma. 

\begin{lemma}\label{eq:diag[p]lemma}
Let $A$ and $B$ be $n\times n$ and $m\times m$ square matrices. Define $N=n+m$ and let $p$ be a positive integer. Then, there exist a square matrix~$X$ of order~$N_p - n_p$ and a permutation matrix $\Gamma$ of order $N_p$ such that 
$\Gamma (\diag(A, B))_{[p]} \Gamma^{-1} = \diag(A_{[p]}, X)$.
\end{lemma}

\begin{proof}
We notice that there exists an $N_p \times N_p$ permutation matrix $\Gamma$ and a function $\zeta \colon \mathbb{R}^n \times  \mathbb{R}^m \to \mathbb{R}^{N_p-n_p}$ such that 
\begin{equation}
\Gamma \begin{bmatrix}
x \\ y 
\end{bmatrix}^{[p]} = \begin{bmatrix}
x^{[p]} \\ \zeta(x, y)
\end{bmatrix}
\end{equation}
for all $x\in\mathbb{R}^n$ and $y\in\mathbb{R}^m$. Now, assume that $x$ and $y$ are vectorial functions and satisfy the differential equations
\begin{equation}\label{eq:dxdt=ax}
\frac{dx}{dt} = Ax    
\end{equation}
and $dy/dt = By$, respectively. Since 
\begin{equation}
\frac{d}{dt}
\begin{bmatrix}
x\\y
\end{bmatrix} = 
\diag(A, B)
\begin{bmatrix}
x\\y
\end{bmatrix}, 
\end{equation}
we have
\begin{equation}\label{eq:nckoi3}
\frac{d}{dt}
\begin{bmatrix}
x\\y
\end{bmatrix}^{[p]}
=
\diag(A, B)_{[p]}
 \begin{bmatrix}
x \\ y
\end{bmatrix}^{[p]}. 
\end{equation}
Multiplying the permutation matrix~$\Gamma$ from the left to this differential equation, we obtain 
\begin{equation}\label{eq:nckoi4}
\frac{d}{dt}
\begin{bmatrix}
x^{[p]}\\ \zeta(x, y)
\end{bmatrix}=
\Gamma
\diag(A, B)_{[p]} \Gamma^{-1}
\begin{bmatrix}
x^{[p]}\\ \zeta(x, y)
\end{bmatrix}. 
\end{equation}

On the other hand, since $x$ satisfies the differential equation~\eqref{eq:dxdt=ax}, we obtain $dx^{[p]}/dt = A_{[p]} x^{[p]}$. Also, since any element of $\zeta(x, y)$ contains at least one element from~$y$, any element of the derivative $d\zeta(x, y)/dt$ must be either zero or contain at least one element from~$y$. Hence, there exists a matrix~$X$ such that $d\zeta(x, y)/dt = X\zeta(x, y)$. From the above differential equations, we obtain 
\begin{equation}\label{eq:nckoi5}
\frac{d}{dt}
\begin{bmatrix}
x^{[p]}\\ \zeta(x, y)
\end{bmatrix}
=
\diag(A_{[p]}, X)
\begin{bmatrix}
x^{[p]}\\ \zeta(x, y)
\end{bmatrix}. 
\end{equation}
Comparing this differential equation and \eqref{eq:nckoi4} completes the proof of the lemma. 
\end{proof}

Let us prove Proposition~\ref{prop:monotonicity}. 

\begin{proof}[Proof of Proposition~\ref{prop:monotonicity}]
Let $W_1$, \dots $W_N \in \mathbb{R}^{m\times m}$ be arbitrary. Define the $m'\times m'$ matrix~$W_i'=\diag(W_i,\,O_{m'-m})$ and let $\hat {\i T}' = Q^\top \otimes I_{(nm')_p} + \bigoplus_{i=1}^N (W_i' \oplus A_i)_{[p]}$. We first claim that the inequality
\begin{equation}\label{eq:firstT'Tineq}
\mu(\hat {\i T}')\geq \mu(\hat {\i T})
\end{equation}  
holds true. Let us prove inequality~\eqref{eq:firstT'Tineq}. Since $W_i' \oplus A_i = \diag(W_i \oplus A_i, I\otimes A_i)$, we have 
\begin{equation}\label{eq:gammatgamma}
\begin{aligned}
&(I_N\otimes \Gamma)\hat {\i T}'(I_N\otimes \Gamma)^{-1}
\\=&    
(I_N Q^\top I_N) \otimes (\Gamma I_{(nm')_p}\Gamma^{-1}) + \bigoplus_{i=1}^N \big( \Gamma(W_i' \oplus A_i)_{[p]}\Gamma^{-1} \big)
\\
=&
Q^\top \otimes  I_{(nm')_p}  + \bigoplus_{i=1}^N 
\big( (W_i \oplus A_i)_{[p]}, X_i \big)
\end{aligned}
\end{equation}
for some square matrices $X_1$, \dots, $X_N$ by Lemma~\ref{eq:diag[p]lemma}. By permuting the rows and columns of the most right hand side of \eqref{eq:gammatgamma}, we see that the matrix~$(I_N\otimes \Gamma)\hat {\i T}'(I_N\otimes \Gamma)^{-1}$ is similar to 
\begin{equation}
\begin{bmatrix}
Q^\top \otimes I+\bigoplus_{i=1}^N \big( W_i \oplus A_i \big)_{[p]}  &O   \\
O & Q^\top \otimes I+\bigoplus_{i=1}^N X_i \end{bmatrix}. 
\end{equation}
Therefore, the matrix $\hat{\i T}'$ is similar to the block diagonal matrix~$\diag(\hat{\i T}, Q^\top \otimes I+\bigoplus_{i=1}^N X_i)$. Hence, we can prove the inequality~\eqref{eq:firstT'Tineq} as 
\begin{equation}
\mu(\hat{\i T}') = \max\left(\mu(\hat{\i T}), \mu\left(Q^\top \otimes I+\bigoplus_{i=1}^N X_i\right)\right) \geq \mu(\hat{\i T}). 
\end{equation}

Now, assume that the switched linear system~$\Sigma_{d, W}$ is stable. Then, the switched linear system~$\Sigma_{d, W'}$ is stable as well by the definition of the matrices~$W_1'$, \dots, $W_N'$. Therefore, the inequality ${\i{U}}_{m'} \geq \mu(\hat {\i T}')$ holds true. This inequality and \eqref{eq:firstT'Tineq} shows ${\i{U}}_{m'} \geq \mu(\hat {\i T})$. Taking the supremum with respect to the set of matrix weights $W_1$, \dots, $W_N$ on the right hand side of this inequality, we obtain the desired inequality, as required. 
\end{proof}

Before closing this section, we give a remark on using complex matrix weights.  

\begin{remark}
Let $p=1$. For complex $m\times m$ matrices $V_1$, \dots, $V_N$, let 
$\hat{\i S} = Q^\top \otimes I_{mn}+\diag(V_1 \oplus A_1,\dotsc,V_N \oplus A_N)$
and define 
\begin{equation}\label{eq:def:UC_m}
{\i{U}}_{\mathbb{C}, m} = \sup_{V_1,\,\dotsc,\,V_N\in \mathbb{C}^{m\times m}} \big \{
\mu(\hat {\i S}) \colon \mbox{$\Sigma_{d, V}$ is stable} \big \}. 
\end{equation}
We claim that
\begin{equation}\label{eq:ucmu2m}
 {\i{U}}_{2m} \geq {\i{U}}_{\mathbb{C}, m}     
\end{equation}
holds true for all $m$. This inequality implies that using the complex weights dose not essentially improve our ability for analyzing the stability property of Markov jump linear system. To prove the inequality~\eqref{eq:ucmu2m}, let $\Re V_i$ and $\Im V_i$ denote the real and imaginary part of $V_i$. Consider the $(2m)\times (2m)$ real matrix weight $W_i$ defined by 
\begin{equation}
W_i =\begin{bmatrix}
\Re V_i& -\Im V_i   \\
\Im V_i& \Re V_i
\end{bmatrix}. 
\end{equation}
Then, we have  
\begin{equation}
\begin{multlined}[.9\linewidth]
\hat{\i T}= 
Q^\top \otimes I_{2nm}+\\\bigoplus_{i=1}^N \Bigg( \begin{bmatrix}
I_m \otimes A_i&O   \\
O&I_m \otimes A_i \end{bmatrix} +\begin{bmatrix}
\Re(W_i) \otimes I_n &-\Im(W_i) \otimes I_n  \\
\Im(W_i) \otimes I_n &\Re(W_i) \otimes I_n \end{bmatrix} \Bigg).
\end{multlined}
\end{equation}
Therefore, $\hat{\i T}$ is similar to the matrix
\begin{equation}\label{eq:thematrix}
\begin{bmatrix}
Q^\top \otimes I_{mn} + \diag(\Re V_i \oplus A_i) & 
-\diag(\Im V_i \oplus A_i)\\
\diag(\Im V_i \oplus A_i)
&
Q^\top \otimes I_{mn} + \diag(\Re V_i \oplus A_i)
\end{bmatrix}
\end{equation}
under a permutation transformation. Furthermore, the set of eigenvalues of the matrix~\eqref{eq:thematrix} contains those of the matrix $Q^\top \otimes I_{mn} + \diag(\Re V_i \oplus A_i) + i\diag(\Im V_i \oplus A_i) = Q^\top \otimes I_{mn} + \diag(V_i \oplus A_i) = \hat{\i S}$. Therefore, we have $\mu(\hat{\i T}) \geq \mu(\hat{\i S})$. This inequality implies 
\begin{equation} \label{ineq:2m:m}
\sup_{W_1,\,\dotsc,\,W_N \in \mathbb{R}^{2m\times 2m}}\{\mu(\hat{\i T})\colon \mbox{$\Sigma_{d, W}$ is stable}\} 
\geq 
\i U_{\mathbb{C}, m}.  
\end{equation}
Hence, to prove inequality~\eqref{eq:ucmu2m}, it is sufficient to show that the stability of $\Sigma_{d, W}$ implies the stability of $\Sigma_{d, V}$. Assume that $\Sigma_{d, W}$ is stable and let us consider the deterministic switched linear system $\Sigma_{d, V} \colon dy/dt=V_{\sigma(t)}y(t)$. Then, we can show that the trajectory~$y$ of the system~$\Sigma_{d, V}$ satisfies the differnetial equation 
\begin{equation}
    \frac{d}{dt}\begin{bmatrix} \Re y \\ \Im y \end{bmatrix} = \begin{bmatrix} \Re \big( V_{\sigma(t)}y \big) \\ \Im \big( V_{\sigma(t)}y \big) \end{bmatrix} = W_{\sigma(t)} \begin{bmatrix} \Re y \\ \Im y \end{bmatrix}.
\end{equation}
Therefore, by the stability of $\Sigma_{d, W}$, we see that there exists a constant $C' > 0$ such that the inequality
\begin{equation}
\lVert y(t) \rVert = \left\lVert \begin{bmatrix} \Re y(t) \\ \Im y(t) \end{bmatrix}  \right\rVert \leq C'
\left\lVert \begin{bmatrix} \Re y(0) \\ \Im y(0) \end{bmatrix}  \right\rVert = C' \lVert y(0) \rVert
\end{equation}
holds true for all $t\geq 0$ and $y(0) = y_0$. This inequality implies the stability of~$\Sigma_{d, V}$, as desired. 
\end{remark}

\section{Spectral optimization} \label{sec:Stochastic}

Although the introduction of the matrix weights $W_i$ allows us to drastically improve our ability to analyze the stability property of Markov jump linear systems as illustrated in Examples~\ref{example:1} and~\ref{example:2}, it is not necessarily trivial to compute the matrix weights achieving the supremum in \eqref{eq:def:U_m}. The first reason is the nonconvexity of the function $\mu$. The other reason is that the set~$\{ W_1, \, \dotsc,\, W_N \colon \mbox{$\Sigma_{d, W}$ is stable}\} \subset \mathbb{R}^{m\times m}$ does not admit a convenient characterization due to the NP-hardness~\cite{Gurvits2009} of checking the stability of deterministic switched linear systems. For these reasons, in this paper, we introduce an alternative quantity by confining our focus on the skew-symmetric matrix weights. Since the set of skew-symmetric matrices of the order, say, $m$, is isomorphic to~$\mathbb{R}^{m(m-1)/2}$, we can employ various optimization techniques for computing the quantity.

Specifically, we define the quantity
\begin{equation}\label{eq:def:V_m}
{\i{V}}_m = \sup_{W_1,\,\dotsc,\,W_N\in \mathbb{R}^{m\times m}}  \{
\mu(\hat {\i T}) \colon \mbox{$W_1$, \dots, $W_N$ are skew-symmetric}\}. 
\end{equation}
Since the matrix exponential of a skew-symmetric matrix is unitary, the deterministic switched linear system $\Sigma_{d, W}$ is stable if $W_1$, \dots, $W_N$ are skew-symmetric. Therefore, the inequality 
\begin{equation}
{\i{V}}_m \leq {\i{U}}_m
\end{equation}
holds true. This inequality and Corollary~\ref{cor:} yield the following instability criterion.   

\begin{corollary}\label{cor:skewsymmetric}
If there exists a positive integer~$m$ such that $\i V_m\geq 0$, then the Markov jump linear system~$\Sigma$ is not exponentially $p$th mean stable. 
\end{corollary}

In the rest of this section, we present a procedure for finding the locally optimal matrix weights $W_1$, \dots, $W_N$ for approximately computing the quantity~$\i V_m$. Let us parametrize the skew-symmetric matrix weights as 
\begin{equation}
W_i = \sum_{\alpha=2}^m \sum_{\beta<\alpha} w_{\alpha \beta}^i R_{\alpha \beta},
\end{equation}
where $w_{\alpha \beta}^i$ is a real variable and $R_{\alpha \beta}$ denotes the $m\times m$ matrix whose entries are all zero except its $(\alpha ,\beta)$th and $(\beta, \alpha)$th entries being $+1$ and $-1$, respectively. We let $E_{ii}$ be the $N \times N$ matrix whose elements are all zero except its $(i, i)$th entry being one. Under this parametrization, we can express the matrix~$\hat{\i T}$ as
\begin{equation}\label{eq;it}
\begin{aligned}
\hat{\i T}
&= Q^\top \otimes I + \bigoplus_{i=1}^N \Big( (I_m \otimes A_i)_{[p]}+ (W_i \otimes I_n)_{[p]} \Big)\\
&= 
\begin{multlined}[t]
Q^\top \otimes I +
\bigoplus_{i=1}^N \big(I_m \otimes A_i \big)_{[p]} \\
+ 
\sum_{i=1}^N \Bigg[\sum_{\alpha=2}^m \sum_{\beta<\alpha} w_{\alpha \beta}^i \Big (E_{ii} \otimes \big( R_{\alpha \beta} \otimes I_n \big)_{[p]} \Big) \Bigg]. 
\end{multlined}
\end{aligned}
\end{equation}
Let 
\begin{equation}
\begin{aligned}
&A_0=Q^\top \otimes I +
\bigoplus_{i=1}^N \big(I_m \otimes A_i \big)_{[p]}, 
\\
&Z_{i \alpha \beta}= E_{ii} \otimes \big(  R_{\alpha \beta} \otimes I_n \big)_{[p]}. 
\end{aligned}
\end{equation}
Then, we can rewrite \eqref{eq;it} as
\begin{equation}\label{eq:affine}
\hat{\i T} = A_0+\sum_{i=1}^N \sum_{\alpha=2}^m \sum_{\beta<\alpha} w_{\alpha \beta}^i Z_{i\alpha \beta}, 
\end{equation}
which shows that the matrix~$\hat{\i T}$ depends linearly on the variables~$w_{\alpha \beta}^i$. Therefore, we can apply gradient-based non-smooth optimization algorithms (e.g., \cite{Burke2002}) to sub-optimally compute the quantity~$\i V_m$. 

\begin{example}
Let us consider the Markov jump linear system~\eqref{eq:MJLS} with
\begin{equation}
A_1=
\begin{bmatrix} 
1.1&1.8   \\
1.75&-0.5 \end{bmatrix}  
,\ 
A_2=\begin{bmatrix} 
-1.1&-2.05   \\
1.95&-0.15 \end{bmatrix} 
\end{equation}
and the $\{1, 2\}$-valued Markov process~$r$ having the infinitesimal generator
\begin{equation}
Q=
\begin{bmatrix} 
-10&10  \\
10&-10 
\end{bmatrix}. 
\end{equation}
Using the MATLAB implementation of the gradient-based non-smooth optimization algorithm~\cite{Curtis2017} and the affine representation~\eqref{eq:affine} of the matrix~$\hat{\i T}$, we obtain $\i V_2 \geq 0.29$. Therefore, Corollary~\ref{cor:skewsymmetric} implies the exponential mean instability of the Markov jump linear system. Notice that the instability criterion given in Proposition~\ref{prop:instability} fails because $\mu({\i T})=-0.07$ is negative.
\end{example}

\section{Conclusion} \label{sec:conclusion}

In this paper, we have proposed a novel criterion for the exponential mean instability of Markov jump linear systems. By considering an auxiliary Markov jump linear system constructed from the Kronecker products of the system matrices and a set of appropriate matrix weights, we have shown that the Markov jump linear system is not exponentially mean stable if a certain matrix parametrized by the matrix weights is not Hurwitz stable. Furthermore, we have confirmed that our instability analysis becomes tighter as we increase the order of the matrix weights. We have also presented a practical procedure based on spectral optimization to find matrix weights for tighter instability analysis. We have finally illustrated the effectiveness of our results by numerical examples.

\section*{References}
\bibliography{library,mendeley_v2}

\end{document}